\documentclass[a4paper,11pt]{article}

\usepackage{amsfonts}
\usepackage{amssymb}
\usepackage{authblk}
\usepackage{amsmath}
\usepackage{xy}
\xyoption{all}

\newtheorem{remark}{Remark}
\newtheorem{proposition}{Proposition}
\newtheorem{example}{Example}

\def\R{{\mathbb R}}

\newcommand{\smalllineskip}{\baselineskip=15pt}
\newcommand{\keywords}[1]{{\footnotesize\hspace{-0.5cm}{\it Keywords: }#1\par
  \vskip.7\baselineskip}}

\newcommand{\MSC}[1]{{\footnotesize\hspace{-0.5cm}{Mathematics Subject Classification 2010: }#1\par
  \vskip.7\baselineskip}}

\renewenvironment{abstract}[0]{
\rm
        \begin{center}\textbf{Abstract}
        \\ \vspace{8pt}
        \begin{minipage}{5.2in}\smalllineskip
        }
        {\end{minipage}\end{center}\vspace{-1pt}}
\newcommand{\emailaddress}[1]{\newline{\sf#1}}
\let\LaTeXtitle\title
\renewcommand{\title}[1]{\LaTeXtitle{\LARGE{\textbf{#1}}}}

\title{Optimal control of affine connection control systems from the point of view of Lie algebroids \thanks{Preprint of an article published in International Journal of Geometric Methods in Modern Physics, 2014 DOI: 10.1142/S0219887814500388 \copyright copyright World Scientific Publishing Company http://www.worldscientific.com/doi/abs/10.1142/S0219887814500388}}
\date{}

\author[1]{L\'igia Abrunheiro}
\author[2]{Margarida Camarinha}

\affil[1]{CIDMA -- Center for Research and Development in Mathematics and Applications, and Higher Institute of Accounting and Administration, University of Aveiro, 3810-500 Aveiro, Portugal \emailaddress{abrunheiroligia@ua.pt}}
\affil[2]{CMUC -- Centre for Mathematics of the University of Coimbra, Department of Mathematics, University of Coimbra, 3001-501 Coimbra, Portugal \emailaddress{mmlsc@mat.uc.pt}}

\begin{document}

\maketitle

\begin{abstract}
The  purpose of this paper is to use the framework of Lie algebroids to study  optimal control problems for affine connection control systems on  Lie groups. In this context, the equations for critical trajectories of the problem are geometrically characterized as a Hamiltonian vector field.
\end{abstract}


\MSC{49J15, 70H05, 70G65, 70E60, 22A22, 53C30, 53B21}

\section{Introduction}

The affine connection control systems (ACCS) constitute a special class of holonomic mechanical systems with inputs. Although not completely general, this kind of simple mechanical system, depending on a Riemannian structure, is the model for many applications of current interest in control. In particular, a large number of systems, such as satellites, hovercrafts and underwater vehicles, have kinetic energy and input forces which are invariant under a certain group action and, moreover, their configuration spaces have a Lie group structure. The aim of this note is to analyze optimal control problems on  Lie groups for this class of control systems.

Optimal control problems (OCP) for affine connection control systems on Riemannian manifolds were studied by M. Barbero-Li\~n\'an and M. C. Mu\~noz-Lecanda in \cite{BarMLecan:2008}, where a presymplectic constraint algorithm was used to characterize the normal and, especially, the abnormal critical trajectories of the problem. Optimal control problems have much in common with variational problems with constraints; in particular, OCP for ACCS are related to  second order variational problems with second order constraints. In this context, these problems have been extensively studied by A. Bloch and P. Crouch  in \cite{BlCr:1993} and, more recently, by  L.  Colombo, D. Martin de Diego and M. Zuccalli (see \cite{Col2010, Col2013b}).

This note reflects our point of view, namely that Lie algebroid formalism is useful for studying this class of mechanical control systems due to the invariance of the  system under the action of the group. To be more precise, we are considering, in the framework of Lie algebroids, an optimal control problem for the Poincar\'{e} representation of an affine connection control system. The geometric description of the problem we develop here follows the formulation of Pontryagin's maximum principle proposed in \cite{Martinez:2007}. The geometry provided by Lie algebroids has already been used in \cite{Barbero-Linan2009a}, for the purpose of characterizing abnormal critical trajectories for this kind of optimal control problem.

\section{Optimal Control of Affine Connection Control Systems on Lie Groups}

Let  $G$  be an n-dimensional Lie group endowed with a left-invariant Riemannian metric $\left\langle \, .\, ,\, .\, \right\rangle$ on $G$. Given $m$ linearly independent elements $E_1,\dots,E_m$ of the Lie algebra $\mathfrak{g}$ of $G$, let $\mathcal{D}$  be the distribution on $G$ generated by the corresponding left-invariant vector fields  $\{E_1^L,\dots,E_m^L\}$, $m\leq n$.

\subsection{Affine connection control systems}

An affine connection control system $(G,\left\langle \, .\, ,\, .\, \right\rangle, \mathcal{D})$ is a simple mechanical control system  without external forces or nonholonomic constraints, and with dynamics given by the equation $\frac{D^2x}{dt^2}(t)=u_a(t)E_a^L(x(t)),$ where the configuration space is the Lie group $G$ and $\frac{D^2x}{dt^2}$ denotes the covariant acceleration  of the state curve $x$ in $G$ defined by the Levi-Civita connection associated  with $\left\langle \, .\, ,\, .\, \right\rangle$. The kinetic energy metric is $\left\langle \, .\, ,\, .\, \right\rangle$, the potential function is zero and the input forces are the left-invariant $1$-forms obtained from the input vector fields $E_1^L,\dots,E_m^L$ via the musical isomorphism defined by the metric tensor. The control functions are $u=(u_1,u_2,\ldots, u_m):[0,T]\rightarrow U$, where $U$ is an open subset of $\R^m$. Throughout this note we make the assumption that the control set $U$ is $\R^m$.

ACCS can be classified as  holonomic mechanical systems with inputs, in contrast with the nonholonomic mechanical systems which have  velocity constraints (see \cite{BlCr:1993} for details). In spite of our particular interest in the study of underactuated systems ($m<n$), we are also considering the fully actuated ones ($m=n$). We assume that our control systems are small-time locally controllable for every proper control set. For these issues see, for instance, \cite{BullLewis:2005}.

The introduction of  new coordinates to the affine connection control system, usually called \emph{quasi-velocities}, is essential  for representing the system. This leads to the notion of \emph{body velocity} of  a curve $x$ in $G$, the curve  $y$ in $\mathfrak{g}$ given by the pullback of the velocity vector field of $x$ by left translation.

The Levi-Civita connection associated  with $\left\langle \, .\, ,\, .\, \right\rangle$ is uniquely defined, via left translation, by the  bilinear map $\alpha: \mathfrak{g}\times \mathfrak{g}\to\mathfrak{g},(y,z)\mapsto\frac12 [y,z]-\frac12\mathbb{I}^{\sharp}(ad_{y}^*\mathbb{I}^{\flat}z+ad_{z}^*\mathbb{I}^{\flat}y)$, where $\mathbb{I}$  is the inner product on the Lie algebra $\mathfrak{g}$ of $G$ corresponding to $\left\langle \, .\, ,\, .\, \right\rangle$ and $\mathbb{I}^{\flat}:\mathfrak{g}\to \mathfrak{g}^*$ and $\mathbb{I}^{\sharp}:\mathfrak{g}^*\to \mathfrak{g}$ are the musical isomorphisms defined by $\mathbb{I}$. Therefore, the affine connection control system may be  represented by the controlled Euler-Poincar\'{e} equations
\begin{equation} \label{eqEulerPoincare}
\frac{dx}{dt}(t)=T_eL_{x(t)}y(t) \quad \quad \quad \frac{dy}{dt}(t)-\mathbb{I}^{\sharp}ad_{y(t)}^*\mathbb{I}^{\flat}y(t)=u_a(t)E_a.
\end{equation}

\subsection{Optimal control problem for an affine connection control system}

Consider the tangent bundle of order two $T^2G$ and its trivialization $\widetilde{T^2G}$, which is the subbundle of $T(G\times\mathfrak{g})$ defined by $\{ (v_x, y, u)\in T_{(x,y)}(G\times\mathfrak{g}):  T_{x}L_{x^{-1}}v_x=y\}$. An element of $\widetilde{T^2G}$ may be represented by the pair $(v_x,u)$, $v_x\in T_xG$, $u \in \mathfrak{g}$, $x\in G$.

The optimal control problem for ACCS is the following. Given $x_0,x_T\in G$,  $v_0\in T_{x_0}G$, $v_T\in T_{x_T}G$,   $T\in \mathbb{R}^+$ and  the cost function  $\mathbb{L}:\widetilde{T^2G}\to \mathbb{R}$, the problem consists of  finding the $C^2$ piecewise-smooth curves $(x,u)$ in $G\times \mathcal{D}_e$ that minimize $\int_0^T\mathbb{L}(v_x(t), u(t))\, dt$ with \mbox{$v_x:[0,T]\to TG$} a $C^2$ piecewise-smooth vector field along $x$, subject to the boundary conditions $x(0)=x_0,$ $x(T)=x_T,$ $v_x(0)=v_0,$  $v_x(T)=v_T$ and satisfying  the affine connection control system $(G,\left\langle \, .\, ,\, .\, \right\rangle, \mathcal{D})$
\begin{equation} \label{eqEulerPoincarev2}
\frac{dx}{dt}=v_x \quad \quad \quad
\frac{d}{dt}(T_{x}L_{x^{-1}}v_x)=\mathbb{I}^{\sharp}ad_{T_{x}L_{x^{-1}}v_x}^*\mathbb{I}^{\flat}(T_{x}L_{x^{-1}}v_x)+u.
\end{equation}

\section{Optimal Control on Lie Algebroids}

Throughout this section, we consider a  Lie algebroid structure $(E,[.,.]_E,\rho_E)$ on a vector bundle $\tau_E:E\to M$, with  anchor map $\rho_E:E\to TM$ and Lie bracket $[.,.]_E$ on the $C^{\infty}(M)$-module $\Gamma(E)$ of sections of $E$.  This structure induces a natural linear Poisson structure on the dual bundle $E^*$ of $E$, with respect to which the dual map $\rho_E^*:E^* \to T^*M$ is Poisson (for the Poisson structure on $T^*M$ induced by the canonical symplectic structure). The corresponding exterior differential  $d^E:\Gamma (\Lambda^kE^*)\to \Gamma (\Lambda^{k+1}E^*)$  defines the corresponding algebroid cohomology. This concept is essential for defining a symplectic structure on $E$, a nondegenerate section of $\Lambda^2E^*$ which  is closed with respect to $d^E$. We refer to \cite{Mackenzie} for a detailed description of Lie algebroids.

The notion of the prolongation of a fibration $\nu: P \to M$ with respect to $E$ is fundamental to the study of optimal control problems on Lie algebroids. We consider, for each point $p\in P$ such that $\nu(p)=x$, the vector space \mbox{$\mathcal{T}_{p}^EP =\{(b,v)\in E_x\times T_pP:\rho(b)=T_p\nu(v)\}.$} The set $\mathcal{T}^{E}P=\bigcup_{p\in P} \mathcal{T}_{p}^E P$ endowed with the following  Lie algebroid structure is called \emph{the prolongation of} $P$ \emph{with respect to} $E$. It is a vector bundle over $P$ with projection $\tau^E_P$ given by $\tau^E_P(p,b,v)=p$, $(b,v)\in \mathcal{T}_{p}^EP$, $p\in P$. The anchor is the map \mbox{$\rho_\mathcal{T}: {\mathcal{T}^E P}\to{TP}$} given by  $\rho_\mathcal{T}(p,b,v)=v$ and the Lie bracket is defined in terms of projectable sections by $[\mathcal{X}_1, \mathcal{X}_2](p)=(p, [\sigma_1, \sigma_2](\nu(p)), [V_1, V_2](p))$, with projectable sections  $\mathcal{X}_i$ defined by $\mathcal{X}_i(p)=(p, \sigma_i(\nu(p)),V_i(p))$,  $i=1,2$. In particular, the prolongation $\mathcal{T}^E E^*$ equipped with a symplectic structure is required to provide the geometrical framework for Pontryangin's maximum principle (PMP). The canonical symplectic form on $\mathcal{T}^E E^*$  is the section  $\Omega\in \Gamma(\Lambda^2(\mathcal{T}^E E^*)^*))$  described by $\Omega=-d^\mathcal{T}\Theta$, with $\Theta$ the Liouville section given by $\Theta_p(b,v)=p(b)$, $(b,v)\in \mathcal{T}_p^E E^*$, $p\in E^*$.

Let the space of controls be the fiber bundle $\pi:{B}\to{M}$, with $M$ the configuration manifold, and the control system on $E$, the section $\sigma:B\to E$ along $\pi$. A trajectory of the system $\sigma$ is an integral curve of the vector field $\rho_E(\sigma)$ along $\pi$. The optimal control problem on a Lie algebroid is the following. Given  $x_0,x_T\in M$,   $T\in \mathbb{R}^+$ and the cost function $L\in\mathcal{C}^{\infty}(B)$ the problem consists of  finding the $C^2$ piecewise-smooth curves $\gamma$ in $B$ that minimize $\int_0^TL(\gamma)\, dt$ subject to the boundary conditions $(\pi \circ \gamma)(0)=x_0$ and $(\pi \circ \gamma)(T)=x_T$ and satisfying  the control system $\frac{d}{dt}(\pi \circ \gamma)=\rho_E(\sigma)(\gamma)$.

The geometric formulation of optimal control problems based on PMP is described as follows. First we introduce  the Hamiltonian function $H\in\mathcal{C}^{\infty}(E^*\times_MB)$ defined by \linebreak $H(p,u)=p(\sigma(u)) -L(u)$. The PMP leads to the Hamiltonian control system $\sigma_H$ (a section of $\mathcal{T}^E{E^*}$ along $\mathrm{pr}_1:{E^*\times_MB}\to{E^*}$) defined on a subset of the manifold $E^*\times_M B$, by means of the symplectic equation $i_{\sigma_H}\Omega=d^{\mathcal{T}}H.$ The integral curves of the vector field $\rho_{\mathcal{T}}(\sigma_H)$ are called the \emph{critical trajectories} of the problem.

\begin{remark}
Let $\beta$ be a section of $\mathcal{T}^E{E^*}$ along
$pr_1:E^*\times_MB\to E^*.$
Note that $i_{\beta}\Omega$ is defined by  $(i_{\beta}\Omega)_{a}(z)=\Omega_p(\beta(a),\mathcal{T} pr_1(z)),$ $z\in \mathcal{T}_a^E{(E^*\times_MB)},$    $a=(p,u)\in E^*\times_MB$,
where
\linebreak $\mathcal{T} pr_1:{\mathcal{T}^E{(E^*\times_MB)}}\to{\mathcal{T}^E{E^*}}$ is the prolongation of $\mathrm{pr}_1$ over the identity map on
$M$.
\end{remark}

Consider $W=\{a\in E^*\times_M B: dH(a)(z)=0\mathrm{\,\, for\,\,\, all \,\,\,}z\in\mathrm{Ker}\,\mathcal{T}{\mathrm{pr}_1}\}$, the set where the symplectic equation has a solution. Then, if there are points in $W$ at which the vector field  $\rho(\sigma_H)$ is not tangent to $W$, it is necessary to implement a constraint algorithm to obtain the submanifold where the critical trajectories lie (see \cite{Martinez:2007} and the references mentioned in this paper).

\section{Description of the Optimal Control Problem on Lie Algebroids}

We consider as configuration manifold the trivialization of $TG$, $G\times \mathfrak{g}$. The space of controls is the trivial vector bundle, ${\pi}:G\times \mathfrak{g}\times \mathcal{D}_e  \to G\times \mathfrak{g}$, $(x,y,u)\mapsto (x,y)$ and the cost function  is the map $L:G\times \mathfrak{g}\times \mathcal{D}_e \to \mathbb{R}$, given by $L(x,y,u)=\mathbb{L}(T_eL_xy,u)$, $y \in \mathfrak{g}, u \in D_e, \, x\in G$. The  boundary conditions required to the curve $(x,y)$ in $G\times \mathfrak{g}$ are $x(0)=x_0, \, x(T)=x_T,$ $y(0)=T_{x_0}L_{x_0^{-1}}v_0,  \, y(T)=T_{x_T}L_{x_T^{-1}}v_T$.

\subsection{The Lie algebroid structure}

We consider the trivial vector bundle $\tau:G\times 3\mathfrak{g}\to G\times \mathfrak{g}$ endowed with the direct product  Lie algebroid structure defined by the action Lie algebroid $G\times \mathfrak{g}\to G$ (associated with the infinitesimal right translation action of $\mathfrak{g}$ on $G$) and the canonical Lie algebroid $2\mathfrak{g}\to \mathfrak{g}$. The anchor map $\rho:G\times 3\mathfrak{g}\to T(G\times \mathfrak{g})$ is given by $\rho(x,y;z,w)=(T_{e}L_{x}z,w).$ The bracket operation is uniquely determined by the Leibniz identity and the bracket of constant sections, given by $[c_{(z_1, w_1)},c_{(z_2, w_2)}]=c_{([z_1,z_2]_{\mathfrak{g}},0)},$  where $c_{(z_1, w_1)}$ and $c_{(z_2, w_2)}$ are  constant sections   defined by $\,(z_1, w_1), (z_2, w_2)\in \mathfrak{g}\times \mathfrak{g}$.

Consider a basis $\{E_i\}_{i=1,\dots, n}$ of $\mathfrak{g}$ obtained by completing the basis of $\mathcal{D}_e$ with  vectors $E_{m+1},\ldots, E_n\in\mathcal{D}_e^\bot$. Denoting the local coordinates in $G$ by $(x^i)$, the local coordinates in $G\times \mathfrak{g}$ with respect to the basis of $\mathfrak{g}$ are denoted by $(x^i,y^i)$. The local coordinates in $G\times 3\mathfrak{g}$, corresponding to the  local basis of sections  $e_{i}=(E_{i},0)\, \, \mbox{and} \, \, e_{i+n}=(0,E_{i}), \, i=1,\dots, n$,  are denoted by $(x^i,y^i;z_i,w_i)$. The corresponding local structure functions $C_{ij}^{k}$, $i, j, k=1,\dots, 2n$ are the constants of structure of $\mathfrak{g}$, if $i, j, k=1,\dots, n$, and are zero, otherwise. The Lie algebra homomorphism $\rho: \Gamma(G\times 3\mathfrak{g})\to \mathfrak{X}(G\times \mathfrak{g})$ is represented  by $\rho(e_i)=(E^L_i,0)$ and $\rho(e_{i+n})=(0,E_i)$, $\, i=1,\dots, n.$

Let  $G\times \mathfrak{g}\times 2\mathfrak{g}^*$ be the dual bundle of $G\times 3\mathfrak{g}$. We denote the local coordinates of an element $p=(x,y;\mu,\xi)$ of $G\times \mathfrak{g}\times 2\mathfrak{g}^*$  by $(x^i,y^i;\mu_{i},\xi_{i})$, corresponding to the local dual basis of sections $e^{i}=(E_{i}^*,0) \, \, \mbox{and} \, \, e^{i+n}=(0,E_{i}^*), \, i=1,\dots, n,$ where $\{E_i^*\}_{i=1,\dots, n}$ is the dual basis of $\{E_i\}_{i=1,\dots, n}$. The dual bundle is endowed with the natural linear Poisson structure $\{\, . \, , \, . \, \}$ given by
\begin{eqnarray}
\nonumber\{f,g\}(p)&=& \frac{\delta g}{\delta x}(p)\left(T_eL_x\frac{\delta f}{\delta \mu}(p)\right)-\frac{\delta f}{\delta x}(p)\left(T_eL_x\frac{\delta g}{\delta \mu}(p)\right) + \frac{\delta g}{\delta y}(p)\left(\frac{\delta f}{\delta \xi}(p)\right)\\ &&-\frac{\delta f}{\delta y}(p)\left(\frac{\delta g}{\delta \xi}(p)\right)+\mu\left([\frac{\delta f}{\delta \mu}(p), \frac{\delta g}{\delta \mu}(p)]\right),
\end{eqnarray}
for each $f,g\in C^{\infty} (G\times \mathfrak{g}\times 2\mathfrak{g}^*)$. Here, $\frac{\delta f}{\delta y}(p)\in \mathfrak{g}^*$  denotes the partial functional derivative defined by $\frac{\delta f}{\delta y}(p)(z)=\frac{\partial f}{\partial y^i}(p)z_i$; $\frac{\delta f}{\delta \mu}(p)\in \mathfrak{g}$ is  the partial functional derivative given by $\frac{\delta f}{\delta \mu}(p)(\xi)=\frac{\partial f}{\partial \mu_i}(p)\xi_i$; and $\frac{\delta f}{\delta x}(p)$ is $ \frac{\partial f}{\partial x^i}(p)dx^i\in T_x^*G$.  Relative to this Poisson bracket and the symplectic bracket on $T^*(G\times \mathfrak{g})$, the dual map $\rho^*:G\times \mathfrak{g}\times 2\mathfrak{g}^* \to T^*(G\times \mathfrak{g})$ given by  $\rho^*(x,y;\mu,\xi)=(x,y;T_x^*L_{x^{-1}}\mu,\xi)$ is Poisson.

The prolongation  of $G\times \mathfrak{g}\times 2\mathfrak{g}^*$ with respect to $G\times 3\mathfrak{g}$ is  the Lie algebroid over $G\times\mathfrak{g}\times 2\mathfrak{g}^*$ identified with $\mathcal{T}=(G\times \mathfrak{g}\times 2\mathfrak{g}^*)\times  (2\mathfrak{g}\times 2\mathfrak{g}^*)$, with anchor $\rho_{\mathcal{T}}$ given by \linebreak $\rho_{\mathcal{T}}(x,y;\mu,\xi;z,w;v_{\mu},v_{\xi})=(T_eL_xz,w;v_{\mu},v_{\xi})$, local basis of sections, $\mathcal{X}_{i}\!=\!(E_{i},0,0,0),$ \linebreak $\mathcal{X}_{i+n}\!=\!(0,E_{i},0,0),$ $\mathcal{V}_i\!=\!(0,0, E_i^*,0),$ $\mathcal{V}_{i+n}\!=\!(0,0,0,E_i^*),$ $i\!=\!1,\dots,n$, and bracket operation given by $[\mathcal{X}_{i},\mathcal{X}_{j}]=C^{k}_{ij}\mathcal{X}_{k}, [\mathcal{X}_{i},\mathcal{V}_j]=0$ and $[\mathcal{V}_{i},\mathcal{V}_j]=0,$ $i,j=1,\ldots, 2n.$ The canonical symplectic  form associated with $G\times 3\mathfrak{g}$ is given,   for each $(x,y;\mu,\xi;z,w;v_{\mu},v_{\xi}),(x,y;\mu,\xi;z^{\prime},w^{\prime};v_{\mu}^{\prime},v_{\xi}^{\prime})\in \mathcal{T}$, by
\begin{eqnarray}
\nonumber \Omega_{(x,y;\mu,\xi)}\left((x,y;\mu,\xi;z,w;v_{\mu},v_{\xi}),(x,y;\mu,\xi;z^{\prime},w^{\prime};v_{\mu}^{\prime},v_{\xi}^{\prime})\right) \\[8pt]
=  v_{\mu}^{\prime}(z) +v_{\xi}^{\prime}(w)-(v_{\mu}(z^{\prime})+v_{\xi}(w^{\prime}))+\mu([z,z^{\prime}]).
 \end{eqnarray}

\begin{remark} \label{rem2}
The corresponding  symplectic form on $G\times \mathfrak{g}\times 2\mathfrak{g}^*$ is defined by
$$\begin{array}{l}\overline{\Omega}_{(x,y;\mu,\xi)}((z_x,w;v_{\mu},v_{\xi}),(z_x^{\prime},w^{\prime};v_{\mu}^{\prime},v_{\xi}^{\prime}))= v_{\mu}^{\prime}(T_xL_{x^{-1}}z_x)+v_{\xi}^{\prime}(w) \\[8pt]
\qquad \qquad  -\left(v_{\mu}(T_xL_{x^{-1}}z_x^{\prime})+v_{\xi}(w^{\prime})\right)+\mu([T_xL_{x^{-1}}z_x,T_xL_{x^{-1}}z_x^{\prime}]).
\end{array}$$
\end{remark}

We also consider the vector  bundle $\pi: G\times \mathfrak{g}\times 2\mathfrak{g}^*\times \mathcal{D}_e\to G\times \mathfrak{g}.$ The prolongation  of $G\times \mathfrak{g}\times 2\mathfrak{g}^*\times \mathcal{D}_e$ with respect to $G\times 3\mathfrak{g}$ is the Lie algebroid identified with $\widetilde{\mathcal{T}}=(G\times \mathfrak{g}\times 2\mathfrak{g}^*\times \mathcal{D}_e)\times  (2\mathfrak{g}\times 2\mathfrak{g}^*\times \mathcal{D}_e),$ with anchor $\rho_{\widetilde{\mathcal{T}}}$ given by $\rho_{\widetilde{\mathcal{T}}}(x,y;\mu,\xi,u;z,w;v_{\mu},v_{\xi},v_u)=(T_eL_xz,w;v_{\mu},v_{\xi}, v_u)$, local basis of sections $\mathcal{X}_{i}\!=\!(E_{i},0,0,0,0)$, $\mathcal{X}_{i+n}\!=\!(0,E_{i},0,0,0)$, $\mathcal{V}_i\!=\!(0,0, E_i^*,0,0)$, $\mathcal{V}_{i+n}\!=\!(0,0,0,E_i^*,0)$, $\mathcal{P}_a\!=\!(0,0,0,0, E_a)$, $i\!=\!1,\dots,n$, $a\!=\!1,\dots,m$ and  bracket operation defined by $[\mathcal{X}_{i},\mathcal{X}_{j}]=C^{k}_{ij}\mathcal{X}_{k}$, $[\mathcal{X}_{i},\mathcal{V}_j]=0,$ $[\mathcal{V}_{i},\mathcal{V}_j]=0,$ $ [\mathcal{P}_{a},\mathcal{X}_{j}]=0,$ $[\mathcal{P}_{a},\mathcal{V}_{j}]=0,$ $[\mathcal{P}_{a},\mathcal{P}_{b}]=0,$ $i,j=1,\dots, 2n, a,b=1,\dots, m$.

\subsection{The Hamiltonian control system}

The control system ${\sigma}$ is the section of $G\times 3\mathfrak{g}$ along ${\pi}:G\times \mathfrak{g}\times \mathcal{D}_e  \to G\times \mathfrak{g}$, defined by ${\sigma}(x,y,u)=(x,y,y,u+\mathbb{I}^{\sharp}ad_{y}^*\mathbb{I}^{\flat}y).$ Note that $\rho(\sigma)$  is exactly the Poincar\'{e} representation (\ref{eqEulerPoincare}) of the affine connection control system $(G, \left\langle \, .\, ,\, .\, \right\rangle, \mathcal{D})$. The  corresponding Hamiltonian function $H:G\times \mathfrak{g}\times 2\mathfrak{g}^*\times \mathcal{D}_e\to \mathbb{R}$ is defined by
\begin{equation}\label{ham}
\displaystyle H(x,y;\mu,\xi,u)=\mu\left(y\right)+\xi\left(u+\mathbb{I}^{\sharp}ad_{y}^*\mathbb{I}^{\flat}y\right)- L(x,y,u).
\end{equation}

\begin{proposition}\label{Prop1}
The critical trajectories for the OCP satisfy the  equations
                                                               \begin{align}
                                                \nonumber  \frac{dx}{dt}= T_eL_xy &\quad &            \frac{dy}{dt}=u+\mathbb{I}^{\sharp}ad_{y}^*\mathbb{I}^{\flat}y\\
   \nonumber  \frac{d\mu}{dt}=T_e^*L_x\frac{\delta L}{\delta x}+ad^*_y\mu &\quad &
\frac{d\xi}{dt}=-\mu+\frac{\delta L}{\delta y}-\mathbb{I}^{\flat}ad_{y}\mathbb{I}^{\sharp}\xi+ad_{\mathbb{I}^{\sharp}\xi}^*\mathbb{I}^{\flat}y\\
0=\xi^{\mathcal{D}_e}-\frac{\delta L}{\delta u}.&\quad &
                                                               \end{align}
Here $\xi^{\mathcal{D}_e}$ denotes $\xi_aE_a^*\in \mathbb{I}^{\flat}\mathcal{D}_e$, for $\xi=\xi_iE_i^*\in \mathfrak{g}^*$.
\end{proposition}

\begin{proof}
Given a section $\beta=\left( z,w;v_{\mu},v_{\xi}\right)$ of  {\small$(G\times \mathfrak{g}\times 2\mathfrak{g}^*)\times  (2\mathfrak{g}\times 2\mathfrak{g}^*)$} the section $ i_{\beta}\Omega$  of $(G\times \mathfrak{g}\times 2\mathfrak{g}^*)\times (2\mathfrak{g}^*\times 2\mathfrak{g})$ is defined by
$$(i_{\beta}\Omega)(p)=\left(-v_{\mu}(p)+ad_{z(p)}^*\mu,-v_{\xi}(p);z(p),w(p)\right), \, p=(x,y;\mu,\xi) \in G\times \mathfrak{g}\times 2\mathfrak{g}^*.$$
On the other hand,
$d^{\widetilde{\mathcal{T}}}H(a)=\left(T_e^*L_x\frac{\delta H}{\delta x}(a),\frac{\delta H}{\delta y}(a); \frac{\delta H}{\delta \mu}(a),\frac{\delta H}{\delta \xi}(a),\frac{\delta H}{\delta u}(a)\right)$, for each \linebreak $a=(x,y;\mu,\xi,u)$. Hence, the Hamiltonian control system  $\sigma_H$, the solution of the symplectic equation $i_{\sigma_H}\Omega=d^{\widetilde{\mathcal{T}}}H$  on  $W=\left\{a\in G\times \mathfrak{g}\times 2\mathfrak{g}^*\times \mathcal{D}_e: \frac{\delta H}{\delta u}(a)=0\right\}$, is given by $\sigma_H(a)=\left(\frac{\delta H}{\delta \mu}(a),\frac{\delta H}{\delta \xi}(a), -T_e^*L_x\frac{\delta H}{\delta x}(a)+ad^*_{\frac{\delta H}{\delta \mu}(a)}\mu,-\frac{\delta H}{\delta y}(a)\right).$ The critical trajectories are  the  integral curves of the vector field $\rho_{\mathcal{T}}(\sigma_H)$ defined by
$$\rho_{\mathcal{T}}(\sigma_H)(a)=\left(T_eL_x\frac{\delta H}{\delta \mu}(a),\frac{\delta H}{\delta \xi}(a),
-T_e^*L_x\frac{\delta H}{\delta x}(a)+ad^*_{\frac{\delta H}{\delta \mu}(a)}\mu,-\frac{\delta H}{\delta y}(a)\right), \,a\in W.$$
\noindent The result follows by straightforward calculations.
\end{proof}

\begin{remark}
If the regularity condition is satisfied, that is, the matrix $\left[\frac{\partial^2 L}{\partial u_a\partial u_b}\right]$ is non-singular, then the equations given by Proposition \ref{Prop1} are Hamiltonian on $G\times\mathfrak{g}\times 2\mathfrak{g}^*$ (Remark \ref{rem2}). In particular, when the system is fully actuated, we recover the second order Euler-Arnold equations for the corresponding Hamiltonian function (see \cite{ColomboMDiego:2011}).
\end{remark}

\section{Minimum Covariant Acceleration Control Problem}

A natural optimal control problem for the affine connection control system is the minimum covariant acceleration control problem with  cost function given by $L(x,y,u)=\frac12\mathbb{I}( u, u)$, $(x,y,u)\in G\times\mathfrak{g}\times\mathcal{D}_e.$ From Proposition \ref{Prop1}, since the regularity condition is guaranteed,  we have the following  Hamiltonian equations (relative to the symplectic structure given in Remark \ref{rem2}),
\begin{align}
\nonumber \frac{dx}{dt}= T_eL_xy &\quad & \frac{dy}{dt}=\mathbb{I}^{\sharp}\xi^{\mathcal{D}_e}+\mathbb{I}^{\sharp}ad_{y}^*\mathbb{I}^{\flat}y\\
\frac{d\mu}{dt}=ad_{y}^*\mu&\quad &
 \frac{d\xi}{dt}=-\mu-\mathbb{I}^{\flat}ad_{y}\mathbb{I}^{\sharp}\xi+ad_{\mathbb{I}^{\sharp}\xi}^*\mathbb{I}^{\flat}y.
\end{align}
When the ACCP is fully actuated, the critical curves give rise to the Riemannian cubic polynomials on a Lie group (see \cite{CrouchSilvaLeite1995,Noa1989}). The study of these curves has been motivated by many applications, namely problems of interpolation in computer graphics, robotics, aeronautics and computational anatomy. Due to that, several approaches to the  geometric description of cubic polynomials have been proposed in the last decade (see, for instance, \cite{AbrCamCGal:2011, CrouchSilvaLeiteCam2000, GayHolmMeirRatiuVialard:2011}).

\begin{example}
A quite well-known example is the following. Let us consider the Lie group \linebreak $G=SO(3)$. The Lie algebra isomorphism identifying  $(so(3), [\,.\,,\,.\,])$ with $(\mathbb{R}^3, \times)$ is denoted by $S$. The left-invariant Riemannian metric on $SO(3)$ is induced by the inner product given by $\mathbb{I}(y,z)=(\mathbb{J}y).z$, $y, z \in \mathbb{R}^3$, where  $\mathbb{J}=\mbox{diag}(\mathbb{J}_1,\mathbb{J}_2,\mathbb{J}_3)$. Given the canonical basis $\{e_1, e_2, e_3\}$ of $\mathbb{R}^3$, let $\mathcal{D}$   be the distribution induced by $\{e_1, e_2\}$. The cost function is defined by \linebreak $L(x,y,u)=\frac12\mathbb{J}_1(u_1)^2+\frac12\mathbb{J}_2(u_2)^2$ and  the control system  is $\dot{x}=xSy \, \wedge \, \mathbb{J}\dot{y}=(\mathbb{J}y)\times y+u_1e_1+u_2e_2,$ $(x,y,u)\in SO(3)\times\mathbb{R}^5$,  $u=(u_1,u_2).$ Note that the pre-control system is properly small-time locally controllable from each $x \in SO(3)$ if $\mathbb{J}_1\neq \mathbb{J}_2$. When $\mathbb{J}_1=\mathbb{J}_2$, we can only guarantee that the system is locally configuration controllable (see \cite{BullLewis:2005}, for details). According to Proposition \ref{Prop1}, the Hamiltonian equations on $(SO(3)\times\mathbb{R}^9, \Omega)$ reduce to
\begin{align}
\nonumber \frac{dx}{dt}=xSy  & & \frac{dy}{dt}=\mathbb{J}^{-1}\xi^{\mathcal{D}_e} +  \mathbb{J}^{-1}(\mathbb{J}y)\times y\\
\frac{d\mu}{dt}=\mu\times y& &
\frac{d\xi}{dt}=-\mu+  \mathbb{J}(\mathbb{J}^{-1}\xi)\times y+(\mathbb{J} y)\times (\mathbb{J}^{-1}\xi),
\end{align}
where the symplectic form $\Omega$ is given by
{\small $$\begin{array}{lll}
 \Omega_{(x,y;\mu,\xi)}\left((z_x,w;v_{\mu},v_{\xi}),(z_x^{\prime},w^{\prime};v_{\mu}^{\prime},v_{\xi}^{\prime})\right)&=&v_{\mu}^{\prime}.(S^{-1}x^{-1}z_x)+v_{\xi}^{\prime}.w-  v_{\mu}.(S^{-1}x^{-1}z_x^{\prime})\\
 &&-v_{\xi}.w^{\prime}+\mu.\left((S^{-1}x^{-1}z_x)\times (S^{-1}x^{-1}z_x^{\prime})\right),
  \end{array}
$$}
for each $z_x, z_x^{\prime}\in T_xSO(3)$, $w, w^{\prime}, v_{\mu}, v_{\mu}^{\prime},v_{\xi}, v_{\xi}^{\prime}\in \R^3$, $x\in SO(3)$  and $y,\mu,\xi\in\mathbb{R}^3$.
\end{example}

\section*{Acknowledgments}

{The research  of L. Abrunheiro was supported by Portuguese funds through the CIDMA -- Center for Research and Development in Mathematics and Applications, and the Portuguese Foundation for Science and Technology (``FCT--Funda\c{c}\~ao para a Ci\^encia e a Tecnologia''), within project PEst-OE/MAT/UI4106/2014. The work of M. Camarinha  was partially supported by the ``Centro de Matem\'atica da Universidade de Coimbra'', funded by the European Regional Development Fund through the program COMPETE and by the Portuguese Government through the FCT -- ``Funda\c{c}\~ao para a Ci\^encia e a Tecnologia'' under the project PEst-C/MAT/UI0324/2013.}

\end{document}